\documentclass{article}

\usepackage[left=0.9in,top=0.9in,right=0.9in,bottom=0.9in]{geometry}
\usepackage{amsmath, amssymb}
\usepackage{mathtools}
\usepackage{hyperref}
\hypersetup{
	colorlinks,
	citecolor=blue,
	filecolor=blue,
	linkcolor=blue,
	urlcolor=blue,
	linktocpage
}
\usepackage{amsthm} 
\usepackage{enumerate}  
\usepackage{url} 
\usepackage{makeidx} 
\usepackage{array}
\usepackage{float}
\usepackage{xcolor}
\usepackage{wrapfig}
\usepackage{graphicx}
\usepackage{rotating}
\usepackage{nicefrac}
\usepackage{longtable}
\usepackage[utf8]{inputenc}
\usepackage[mathlines]{lineno}
\usepackage{color}
\definecolor{red}{rgb}{.8,0,0}

\definecolor{bblu}{rgb}{0,0,1}

\definecolor{purple}{rgb}{.8,0,1}

\newtheorem{theorem}{Theorem}[section]

\newtheorem{lemma}[theorem]{Lemma} 
\newtheorem{claim}[theorem]{Claim}

\newtheorem{question}[theorem]{Question}

\theoremstyle{definition}
\newtheorem{defn}{Definition}

\newcommand\comp[1]{{\mkern2mu\overline{\mkern-2mu#1}}}

\DeclareMathOperator\tr{tr}

\DeclareMathOperator\sument{sum}
\DeclareMathOperator\Aut{Aut}

	\newcommand\cH{{\mathcal H}}

	\newcommand{\lam}{\lambda}
	\newcommand{\f}[2]{\frac{#1}{#2}}
	
	\title{New Eigenvalue Bound for the  Fractional Chromatic Number}
	\author{Krystal Guo\thanks{Korteweg-de Vries Institute for Mathematics, University of Amsterdam, Amsterdam, the Netherlands. {\tt k.guo@uva.nl}.} \and Sam Spiro\thanks{Dept.\ of Mathematics, Rutgers University, Piscataway NJ, United States of America. {\tt sas703@scarletmail.rutgers.edu}. This material is based upon work supported by the National Science Foundation Mathematical Sciences Postdoctoral Research Fellowship under Grant No. DMS-2202730.}}
	\date{\today}
	
	\renewcommand{\SS}[1]{} 
	\newcommand\krystalsays[1]{}
	\newcommand\samsays[1]{}
	\newcommand{\norm}[1]{\|#1\|}
	\renewcommand{\vec}[1]{\textbf{#1}}

	\newtheoremstyle{LemmaNum}
	{\topsep}{\topsep}
	{\itshape}
	{}
	{\bfseries} 
	{.}
	{ }
	{\thmname{#1}\thmnote{ \bfseries #3}}
	\theoremstyle{LemmaNum}
	\newtheorem{lemn}{Lemma}
	\newtheorem{thmn}{Theorem}

	\newcommand{\ratio}[1]{\frac{\lam_{\max}(#1)}{|\lam_{\min}(#1)|}}
	
\begin{document}
		
		\maketitle
		
		\begin{abstract}
			Given a graph $G$, we let $s^+(G)$ denote the sum of the squares of the positive eigenvalues of the adjacency matrix of $G$, and we similarly define $s^-(G)$. We prove that
			\[\chi_f(G)\ge 1+\max\left\{\frac{s^+(G)}{s^-(G)},\frac{s^-(G)}{s^+(G)}\right\}\]
			and thus strengthen a result of Ando and Lin, who showed the same lower bound  for the chromatic number $\chi(G)$.  
			We in fact show a stronger result wherein we give a bound using the eigenvalues of $G$ and $H$ whenever $G$ has a homomorphism to an edge-transitive graph $H$. Our proof utilizes ideas motivated by association schemes.
			\vspace{1em}

			
		\end{abstract}
		
		\section{Introduction}
		Given a graph $G$, we say that $\lambda$ is an eigenvalue of $G$ if it is an eigenvalue for the adjacency matrix of $G$. There is a long history of using the eigenvalues of a graph $G$ to bound its chromatic number.  Early work on chromatic numbers and graph spectra includes a result of Wilf \cite{Wil1967} which gives an upper bound on $\chi(G)$ using the largest eigenvalue of $G$.  Recent variants of this upper bound have been proven for 
		digraphs \cite{Moh2010},  hypergraphs \cite{Coo2012}, and simplicial complexes \cite{Gol2017}. 
		A spectral lower bound for $\chi(G)$ was  established by Hoffman \cite{Hof1970} in terms of the largest and smallest eigenvalues of $G$ (see \eqref{eq:Hoffman} below), and it is related to an upper bound for the independence number in the setting of association schemes by Delsarte \cite[Section 3.3]{Del1973}. 
		Many extensions of these results have been established; for example, see \cite{Bil2006,Nik2007,vDaSot2016}. 
		
		Using the numbers of positive and negative eigenvalues, Cvektovi\'{c}~\cite{Cve1972} proved a lower bound for $\chi(G)$ (see also \cite{ElpWoc2017}).  In a similar spirit, Ando and Lin~\cite{AndLin2015} proved a lower bound for $\chi(G)$ in terms of the squares of the positive and negative eigenvalues of $G$, which will be of particular importance to this work.  To formally state their result, we let $\lambda_i$ be the $i$th largest eigenvalue of a graph $G$ and define
		\[
		s^+(G) = \sum_{i: \lambda_i >0 } \lambda_i^2, \qquad 		
		s^-(G) = \sum_{j: \lambda_j < 0 } \lambda_j^2.
		\]
		When the context is clear, we will omit $G$ from this notation. The quantities $s^+$ and $s^-$ were first considered by Wocjan and Elphick \cite{WocElp2012}.  They conjectured the following, which was proven by Ando and Lin \cite{AndLin2015}. 
		\begin{theorem}[\cite{AndLin2015}]\label{thm:AL}
			For any graph $G$, we have
			\[\chi(G)\ge 1+\max\left\{\frac{s^+(G)}{s^-(G)},\frac{s^-(G)}{s^+(G)}\right\}.\]
		\end{theorem}
		Theorem~\ref{thm:AL} has inspired a large amount of follow up work, including extensions to quantum chromatic numbers \cite{ElpWoc2019},
		quantum graphs \cite{Gan2021}, and similar bounds for $p$-norms of certain classes of matrices \cite{MaoLiu2021}.
		
		In this paper, we show that Theorem \ref{thm:AL} also holds when $\chi(G)$ is replaced by the fractional chromatic number $\chi_f(G)$ (whose definition is recalled below). 
		We note that Anekstein, Elphick, and Wocjan \cite{wocjan2018more} asked whether Theorem \ref{thm:AL} can be strengthened by replacing the chromatic number $\chi(G)$ with the vector chromatic number $\chi_c(G)$, which satisfies the inequality $\chi(G)\ge \chi_c(G)$.  Our result can be viewed as the first progress towards this question;, it is well-known that these various chromatic numbers satisfy the inequalities $\chi(G) \geq \chi_f(G) \geq \chi_c(G)$, see for example \cite{Sch1979}.
		
		Recall that a \textit{homomorphism} between two graphs $G,H$ is a map $\phi:V(G)\to V(H)$ such that $\phi(u)\phi(v)\in E(H)$ whenever $uv\in E(G)$.  We define the \textit{Kneser graph} $K_{n;k}$ to be the graph whose vertices are the $k$-subsets of an $n$-element set, where two $k$-subsets are adjacent if and only if they are disjoint. The \textit{fractional chromatic number} of a graph $G$ is then defined by
		\[\chi_f(G)=\inf_{(n,k)} \frac{n}{k},\]
		where the infimum runs over all pairs $(n,k)$ such that there exist a homomorphism from $G$ to $K_{n;k}$. We note that one can equivalently define the fractional chromatic number as the optimum of the linear programming relaxation for an integer programming formulation of chromatic number.  For more background on the fractional chromatic number, we refer to  \cite{Sch2011}.   With $\chi_f(G)$ defined we can now state our strengthening of Theorem~\ref{thm:AL}.

		\begin{theorem}\label{thm:main}
			For any graph $G$, we have
			\[\chi_f(G)\ge 1+\max\left\{\frac{s^+(G)}{s^-(G)},\frac{s^-(G)}{s^+(G)}\right\}.\]
		\end{theorem}

		Our proof of Theorem \ref{thm:main} uses ideas motivated by association schemes.  By generalizing our approach, we also prove the following result, which turns out to be somewhat stronger than Theorem~\ref{thm:main}.

		\begin{theorem}\label{thm:mainGen}
			If there exists a homomorphism from a graph $G$ to an edge-transitive graph $H$, then 
			\[\frac{\lambda_{\max}(H)}{|\lambda_{\min}(H)|}\ge \max\left\{\frac{s^+(G)}{s^-(G)},\frac{s^-(G)}{s^+(G)}\right\},\]
			where $\lambda_{\max}(H),\lam_{\min}(H)$ are the largest and smallest eigenvalue of the adjacency matrix of $H$, respectively.
		\end{theorem}
		
		As an aside, Theorem \ref{thm:mainGen} can be used to show the non-existence of homomorphisms between graphs. For example, it shows that the Petersen graph $P$ has no homomorphism into $C_7$, since $\nicefrac{s^+(P)}{s^-(P)} = \nicefrac{8}{7} = 1.\overline{142857}$ and $\nicefrac{\lambda_{\max}(C_7)}{|\lambda_{\min}(C_7)|} = 1.109916$, correct to the first six decimals.

		\subsection{Tightness and comparisons with other bounds}\label{sec:comparison}
		Here we briefly discuss examples showing that our bounds are tight, as well as how our bounds compare to other known lower bounds for $\chi_f(G)$
		
		Theorem~\ref{thm:AL} (and hence the stronger Theorems~\ref{thm:main} and \ref{thm:mainGen}) is tight for a number of examples.  The complete graph $K_n$ has $s^+(K_n) = (n-1)^2$ and  $s^-(K_n) = n-1$, so Theorem \ref{thm:AL} holds with equality. For a bipartite graph with $m$ edges, we have $s^+ = s^- = m$ and the bound also holds with equality here.  A less trivial example includes the Paley graph $P_9$ on $9$ vertices (see Figure~\ref{fig:examples}), which has chromatic number 3 and spectrum $\{4^{(1)}, 2^{(4)}, (-2)^{(4)} \}$ (where the multiplicities are shown in superscript).  We note that the clique and Paley graph examples can also be shown through Theorem~\ref{thm:mainGen} by considering the homomorphisms from these graphs to themselves.
		
		\begin{figure}[htb]
			\centering
			\includegraphics[scale=0.6]{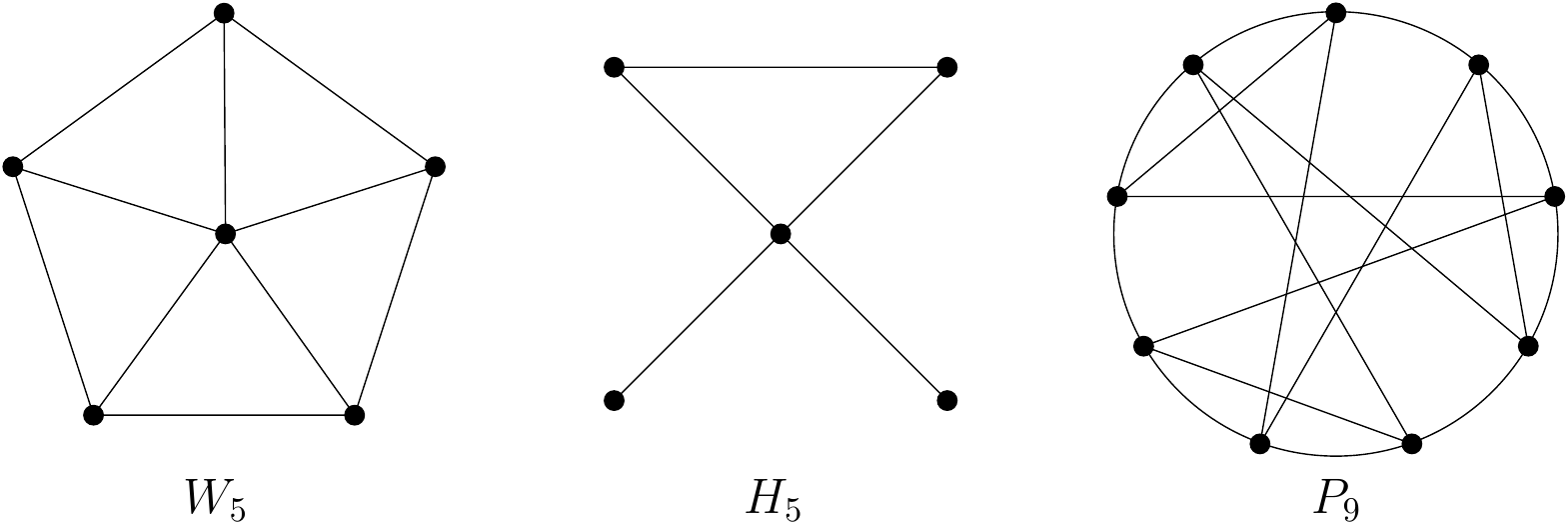}
			\caption{\textit{From left to right:} Graphs $W_5$ and $H_5$, which are witnesses to the comparability of Theorem \ref{thm:main} with the clique number and with the Hoffman bound, respectively, and $P_9$ the Paley graph of order $9$. }
			\label{fig:examples}
		\end{figure}
		
		In addition to these examples, there are several other ways that one can view Theorems~\ref{thm:main} and \ref{thm:mainGen} as being tight.  For example, Theorem~\ref{thm:mainGen} can not be strengthened to hold whenever $G$ has a homomorphism to an $H$ which is vertex-transitive.  A counterexample to such a statement follows by considering $G=K_3$ and $H=\overline{C_6}$ (which is vertex-transitive but not edge-transitive).  As $\overline{C_6}$ contains a triangle, there exists a homomorphism from $K_3$ to $\overline{C_6}$, but
		\[\frac{3}{2}=\frac{\lambda_{\max}(\overline{C_6})}{|\lambda_{\min}(\overline{C_6})|}<1+\max\left\{\frac{s^+(K_3)}{s^-(K_3)},\frac{s^-(K_3)}{s^+(K_3)}\right\}=3,\]
		showing that an extension of Theorem~\ref{thm:mainGen} to vertex-transitive graphs can not hold.
		
		We next consider how  Theorem~\ref{thm:main} compares with other well known lower bounds for $\chi_f(G)$.  Unlike the chromatic number, relatively few spectral bounds are known for $\chi_f(G)$.  One such bound is Hoffman's bound, which says
		\begin{equation}\chi_f(G)\ge 1+ \frac{\lambda_{\max}(G)}{|\lambda_{\min}(G)|}\label{eq:Hoffman}.\end{equation}
		
		In general the bounds of Theorem~\ref{thm:main} and \eqref{eq:Hoffman} are incomparable.  For example, \eqref{eq:Hoffman} gives $\chi_f(G)\ge \nicefrac{5}{2}$ when $G$ is the Petersen graph (which is tight), while Theorem~\ref{thm:main} only gives $\chi_f(G)\ge \nicefrac{11}{3}$. On the other hand, a small computation in SageMath \cite{sage} gives that amongst the $11,855$ graphs on $5, 6, 7, 8$ vertices which are connected and non-bipartite, there are $11,014$  graphs $G$ for which Theorem \ref{thm:main} gives a better bound on the fractional chromatic number than \eqref{eq:Hoffman}.  A concrete example is $H_5$, shown in the middle of Figure \ref{fig:examples}, which has
		\[
		1+\max\left\{\frac{s^+(H_5)}{s^-(H_5)},\frac{s^-(H_5)}{s^+(H_5)}\right\} \approx 2.331454 ,  \qquad 1+ \frac{\lam_{\max}(H_5)}{|\lam_{\min}(H_5)|} \approx 2.291859,
		\]
		rounded to the sixth digit. Thus Theorem~\ref{thm:main} and \eqref{eq:Hoffman} are incomparable. 
		
		Another well known bound is $\chi_f(G)\ge \omega(G)$, where $\omega(G)$ is the size of the largest clique of $G$.  Out bound is also incomparable with $\omega(G)$: the wheel $W_5$ (see left side of Figure \ref{fig:examples}) has 
		\[
		3 = \omega(W_5) > 1+\max\left\{\frac{s^+}{s^-},\frac{s^-}{s^+}\right\} \approx 2.725877
		\]
		and the $5$-cycle $C_5$ has
		\[
		2 = \omega(C_5) < 1+\max\left\{\frac{s^+}{s^-},\frac{s^-}{s^+}\right\}
		\approx 2.099106. 
		\]

		\subsection{Organization and notation}			
		
		We prove Theorem~\ref{thm:mainGen} in Section \ref{sec:proofofmain} assuming a technical lemma, and then show how this implies Theorem~\ref{thm:main}.  In Section \ref{sec:prooflemma} we prove this technical lemma, thereby completing the proof.  We conclude with some open problems in Section \ref{sec:conclusion}.

		\textbf{Notation}.  Given real matrices $X=(x_{i,j})_{i,j}$ and $Y=(y_{i,j})_{i,j}$, we let $\sument(X)=\sum_{i,j} x_{i,j}$. 
		We denote the Schur product with $\circ$ and recall that $X\circ Y=(x_{i,j}y_{i,j})_{i,j}$. 
		We use the following standard inner product for real $m\times n$ matrices $X$ and $Y$:
		\[
		\langle X, Y \rangle := \tr(XY^T) = \sument(X\circ Y)
		\]
		where the second equality is a standard property of matrix multiplication. We will sometimes write $\norm{X}^2$ for $\langle X, X\rangle$.
		Note that if $X$ is symmetric, then $\norm{X}^2$ is equal to the sum of the squares of the eigenvalues of $X$.  Throughout $J$ denotes the all 1's matrix.
		
		\section{Partitions into fibres: proof of main results}\label{sec:proofofmain}
		
		Before going into the details of the proof, we briefly overview the ideas of the argument.  Consider the spectral decomposition of the adjacency matrix $A=\sum \mu_i E_i$, where $E_i$ is the idempotent projector onto the $\mu_i$-eigenspace. We may write $A=X-Y$ where 
		\[ X=\sum_{i:\mu_i> 0} \mu_i E_i \, \text{ and }\, Y=-\sum_{i:\mu_i< 0} \mu_i E_i.\] 
		We observe that $X$ and $Y$ are both positive semidefinite matrices and that $\norm{X}^2=s^+,\norm{Y}^2=s^-$.  Thus proving Theorem~\ref{thm:mainGen} is equivalent to showing 
		\[ \norm{X}^2\le \f{\lam_{\max}(H)}{|\lam_{\min}(H)|}\norm{Y}^2 \text{\ \ \ and\ \ \ } \norm{Y}^2\le \f{\lam_{\max}(H)}{|\lam_{\min}(H)|}\norm{X}^2
		\]
		whenever there exists a homomorphism $\phi:V(G)\to V(H)$.  Given such a map $\phi$, we can naturally partition $V(G)$ into the sets $\phi^{-1}(u)$ with $u\in V(H)$. With this in mind, we make the following definition.
		
		\begin{defn}
			Given a graph $G$, we say that $\bigsqcup_{u\in V(H)} V_u$ is an \textit{$H$-partition} of $G$ if the sets $V_u\subseteq V(G)$ partition $V(G)$.  Given an $H$-partition, if $X$ is a matrix whose rows and columns are indexed by $V(G)$, we write $X_{[u,v]}$ to be the submatrix consisting of the rows indexed by $V_u$ and the columns indexed by $V_v$.
		\end{defn}
		
		We now state our main lemma.
		\begin{lemma}\label{lem:main}
			Let $G$ be a graph with an $H$-partition $\bigsqcup_{u\in V(H)} V_u$, and let $X$ be a PSD matrix with rows and columns indexed by $V(G)$.  If $H$ is vertex and edge-transitive, then
			\[\norm{X}^2\le \left(1+\frac{\lam_{\max}(H)}{|\lam_{\min}(H)|}\right)\sum_{(u,v):u,v\in V(H),\{u,v\}\notin E(H)} \norm{X_{[u,v]}}^2.\]
		\end{lemma}
		For the moment we postpone the proof of Lemma~\ref{lem:main} and show how it implies Theorem~\ref{thm:main}.  We note that our approach for the rest of this section will closely follow that of Ando and Lin \cite{AndLin2015}.  
		\begin{lemma}\label{lem:confomal}
			Let $G$ be a graph with an $H$-partition $\bigsqcup_{u\in V(H)} V_u$, and let $X,Y$ be real PSD matrices with rows and columns indexed by $V(G)$.  If $H$ is vertex and edge-transitive, and if $XY=0$ and $X_{[u,v]}=Y_{[u,v]}$ whenever $\{u,v\}\notin E(H)$, then \[\norm{X}^2\le \frac{\lam_{\max}(H)}{|\lam_{\min}(H)|}\norm{Y}^2.\]
		\end{lemma}
		\begin{proof}
			For ease of notation, we omit writing that our sums are over all ordered pairs $(u,v)$ with $u,v\in V(H)$. Let $X,Y$ be as in the hypothesis of the lemma.  Because $XY=0$, we have  $\norm{X+Y}^2=\norm{X-Y}^2$, which by hypothesis on $X,Y$ is equivalent to
			\[4\sum_{\{u,v\}\notin E(H)}\norm{X_{[u,v]}}^2+\sum_{\{u,v\}\in E(H)} \norm{X_{[u,v]}+Y_{[u,v]}}^2=\sum_{\{u,v\}\in E(H)}\norm{X_{[u,v]}-Y_{[u,v]}}^2.\]
			This is equivalent to 
			\begin{align*}4\sum_{\{u,v\}\notin E(H)}\norm{X_{[u,v]}}^2&=\sum_{\{u,v\}\in E(H)}\tr((X_{[u,v]}-Y_{[u,v]})(X_{[u,v]}-Y_{[u,v]})^T)-\tr((X_{[u,v]}+Y_{[u,v]})(X_{[u,v]}+Y_{[u,v]})^T)\\&=-4\sum_{\{u,v\}\in E(H)} \tr(X_{[u,v]}Y_{[u,v]}^T).\end{align*}
			
			Let $d=\sum_{\{u,v\}\notin E(H)}\norm{X_{[u,v]}}^2$.  The equality above together with two applications of Cauchy-Schwarz gives
			\begin{align*}d&=-\sum_{\{u,v\}\in E(H)} \tr(X_{[u,v]}Y_{[u,v]}^T)\\&\le \sum_{\{u,v\}\in E(H)} \norm{X_{[u,v]}}\cdot \norm{Y_{[u,v]}}\\&\le \left(\sum_{\{u,v\}\in E(H)} \norm{X_{[u,v]}}^2\right)^{1/2}\left(\sum_{\{u,v\}\in E(H)} \norm{Y_{[u,v]}}^2\right)^{1/2}\\&=(\norm{X}^2-d)^{1/2}(\|Y\|^2-d)^{1/2}.\end{align*}
			Thus
			\[d^2\le (\norm{X}^2-d)(\|Y\|^2-d),\]
			which is equivalent to
			\begin{equation}(\norm{X}^2+\|Y\|^2)d\le \norm{X}^2\|Y\|^2.\label{eq:normSum}\end{equation}
			By Lemma~\ref{lem:main}, we have
			\[\norm{X}^2\le \left(1+\frac{\lam_{\max}(H)}{|\lam_{\min}(H)|}\right)d.\]
			Combining this with \eqref{eq:normSum} implies
			\[\norm{X}^2+\norm{Y}^2\le \left(1+\frac{\lam_{\max}(H)}{|\lam_{\min}(H)|}\right)\norm{Y}^2,\]
			giving the desired result.
		\end{proof}
		We can now prove our main theorem, which we restate below. 
		
		\begin{thmn}[\ref*{thm:mainGen}]
			If there exists a homomorphism from a graph $G$ to an edge-transitive graph $H$, then 
			\[\frac{\lambda_{\max}(H)}{|\lambda_{\min}(H)|}\ge \max\left\{\frac{s^+(G)}{s^-(G)},\frac{s^-(G)}{s^+(G)}\right\},\]
			where $\lambda_{\max}(H),\lam_{\min}(H)$ are the largest and smallest eigenvalue of the adjacency matrix of $H$, respectively.
		\end{thmn}
		\begin{proof}
			If $H$ is edge-transitive but not vertex-transitive, then $H$ is bipartite by \cite[Lemma 3.2.1]{GR}.  The existence of a homomorphism $\phi:V(G)\to V(H)$ implies that $G$ is also bipartite. Since the spectrum of bipartite graphs are symmetric about $0$ on the real line, we have $s^+(G)=s^-(G)$ and $\lam_{\max}(H)=-\lam_{\min}(H)$ and the result is trivial.  Thus from now on we may assume $H$ is both vertex and edge-transitive.
			
			For each $u\in V(H)$, define $V_u=\phi^{-1}(u)$. Note that $\sqcup_u V_u$ is an $H$-partition.  Consider the spectral decomposition of the adjacency matrix $A=\sum \mu_i E_i$, where $E_i$ is the idempotent eigenprojector on to the $\mu_i$-eigenspace. We write $A=X-Y$ where $X=\sum_{i:\mu_i> 0} \mu_i E_i$ and $Y=-\sum_{i:\mu_i< 0} \mu_i E_i$.  Note that $X,Y$ are both real PSD matrices.  Since $\phi$ is a homomorphism, there are no edges between $V_u$ and $V_v$ in $G$ if $\{u,v\}\notin E(H)$, so $A_{[u,v]}=0$ in this case.  Since $A=X-Y$, this implies $X_{[u,v]}=Y_{[u,v]}$ whenever $\{u,v\}\notin E(H)$.  By Lemma~\ref{lem:confomal} we conclude \[\|X\|^2\le \frac{\lam_{\max}(H)}{|\lam_{\min}(H)|}\|Y\|^2.\]  Since $\|X\|^2=s^+(G)$ and $\|Y\|^2=s^-(G)$, this gives $\frac{\lam_{\max}(H)}{|\lam_{\min}(H)|}\ge s^+(G)/s^-(G)$, and a completely analogous proof gives the lower bound of $s^-(G)/s^+(G)$.
		\end{proof}
		
		We now show how Theorem~\ref{thm:mainGen} implies the analogous lower bound for the fractional chromatic number.
		
		\begin{proof}[Proof of Theorem~\ref{thm:main}]
			We first claim that the Kneser graphs $K_{n;k}$  satisfy \[\frac{\lam_{\max}(K_{n;k})}{|\lam_{\min}(K_{n;k})|}=\frac{n}{k}-1.\]
			Indeed, it is well known that the eigenvalues of the Kneser graph $K_{n;k}$ can be written as
			\[(-1)^i \binom{n-k-i}{k-i}\]
			for all $0\le i\le k$.  In particular, $\lam_{\max}(K_{n;k})=\binom{n-k}{k}$ and $\lam_{\min}=-\binom{n-k-1}{k-1}=-\frac{k}{n-k} \binom{n-k}{k}$, giving the claim.
			
			It is not difficult to see that the Kneser graphs are edge-transitive.  Thus if we let $\Phi(G)$ (respectively $\mathcal{K}(G)$) denote the set of edge-transitive graphs (respectively kneser graphs) $H$ such that there exists a homomorphism $\phi:V(G)\to V(H)$, then Theorem~\ref{thm:mainGen} implies
			\[\max\left\{\frac{s^+(G)}{s^-(G)},\frac{s^-(G)}{s^+(G)}\right\}\le \inf_{H\in \Phi(G)} \ratio{H}\le \inf_{K_{n;k}\in \mathcal{K}(G)} \ratio{K_{n;k}}=\inf_{K_{n;k}\in \mathcal{K}(G)} \frac{n}{k}-1=\chi_f(G)-1,\]
			giving the result.
		\end{proof}
		
		\section{Proof of Lemma~\ref{lem:main}} \label{sec:prooflemma}
		
		We begin by establishing the key property we need about edge-transitive graphs. For this we let $\Aut(H)$ denote the automorphism group of $H$, i.e.\ the group of permutations on $V(H)$ which are graphs isomorphisms.  For ease of notation, we write $\{u,v\}$ as $uv$ and $\{\pi(u),\pi(v)\}$ as $\pi(uv)$.
		
		\begin{lemma}\label{lem:transitive}
			If $H$ is edge-transitive, then there exists a set of non-empty graphs $\cH=\{H_1,\ldots,H_n\}$ on $V(H)$ satisfying the following properties:
			\begin{enumerate}[(a)]
				\item $H\in \cH$;
				\item For every pair of distinct vertices $u,v\in V(H)$, there exists a unique $i$ with $uv\in V(H_i)$;
				\item For every pair of edges $uv,xy\in E(H_i)$, there exists $\pi\in \Aut(H)$ with $\pi(uv)=xy$; and
				\item For all $i$, the restriction of any $\pi\in \Aut(H)$ to $H_i$ is an automorphism for $H_i$.
			\end{enumerate}
		\end{lemma}
		\begin{proof}
			
			We consider the action of $\Aut(H)$ on the set of unordered pair of vertices of $H$ and we take $O_1,\ldots O_n$ to be the orbits of this action. Here, the orbit containing a pair $uv$ is  $O =\{xy:\exists \pi\in \Aut(H),\ \pi(uv)=xy\}$.  Define $H_i$ to be the graph on $V(H)$ with edge set $O_i$ and let $\cH:=\{H_1,\ldots,H_m\}$.

			Since $H$ is edge-transitive, $E(H)$ must be an orbit, so we have $H\in \cH$.  Since the orbits partition the pairs of $V(H)$, we have that the sets $\{E(H_i)\}_{i=1}^n$ partition the pairs. 	Since orbit partitions give rise to systems of blocks of imprimitivity, the other properties follow. 
		\end{proof}
		
		As an aside, it is known that if $H$ is a Kneser graph, then $\Aut(H)$ is the set of permutations of $V(H)$ induced by permuting the underlying ground set.  Using this, one can verify that when $H$ is a Kneser graph, the family $\cH$ from Lemma~\ref{lem:transitive} is the set of graphs known as the Johnson scheme, which is one of the most famous examples of an association scheme.  As such, one can view the family $\cH$ from Lemma~\ref{lem:transitive} as an analog of association schemes which exist for all edge-transitive graphs. For more on association schemes we refer the reader to the books \cite{bannai2021algebraic,godsil2016erdos}.  However, we emphasize that our approach, which was originally motivated by association schemes, requires no knowledge of association schemes to understand the details of our arguments.
		
		For the rest of this section, we fix $H$ to be a vertex and edge-transitive graph with $\cH=\{H_1,\ldots,H_n\}$ the family guaranteed by Lemma~\ref{lem:transitive}, and without loss of generality we assume $H_n=H$.  We let $A_i$ denote the adjacency matrix of $H_i$, and for notational convenience we let $A_0$ denote the identity matrix of dimension $|V(H)|$.  Note that Lemma~\ref{lem:transitive}(b) implies $\sum_{i=0}^n A_i=J$.
		
		It turns out that to prove Lemma~\ref{lem:main} for the matrix $X$, it suffices to prove an analogous inequality for the matrix $Z$ defined by $Z_{u,v}:=||X_{[u,v]}||^2$ for $u,v\in V(H)$.  Note that $Z$ is non-negative, and it will turn out to be PSD whenever $X$ is.  The following lemma establishes this analogous inequality for $Z$ in the special case that $Z$ is in the span of the $A_i$ matrices.
		
		\begin{lemma}\label{lem:ZJohnson}
			Let $z_0,z_1,\ldots,z_n$ be real numbers such that $Z:=\sum_{k=0}^n z_i A_i$ is a non-negative PSD matrix.  Then 
			\[\sument(Z)  \leq  \left(1+\frac{\lam_{\max}(A_n)}{|\lam_{\min}(A_n)|}\right)\sument((J-A_n) \circ Z).\]
		\end{lemma}
		
		\begin{proof}
			Because the graphs $H_i$ are edge disjoint, $Z=\sum_{k=0}^n z_i A_i$ will be non-negative if and only if $z_i\ge 0$ for all $i$.  By using Raleigh quotients, we see $\vec{x}^T A_i \vec{x}\le \lam_{\max}(A_i)$ for any real vector $\vec{x}$ and $0\le i\le n$.  In particular, if $\vec{x}$ is the eigenvector of $A_n$ associated to $\lam_{\min}(A_n)$, then $Z$ being PSD implies
			\[0\le \vec{x}^T Z\vec{x}=\sum_{i=0}^{n}z_i\cdot \vec{x}^T A_i\vec{x} \le \sum_{i=0}^{n-1} z_i \lam_{\max}(A_i)+z_n\lam_{\min}(A_n),\]
			where this last step implicitly used $z_i\ge 0$ for all $i$.  Rearranging gives
			\begin{equation}\sum_{i=0}^{n-1} z_i \lam_{\max}(A_i)\ge -z_n\lam_{\min}(A_n)=z_n|\lam_{\min}(A_n)|,\label{eq:rearrange}\end{equation}
			where the last step used that the smallest eigenvalue of a real symmetric matrix with 0 diagonals is non-negative.
			
			Recall that we assumed $H=H_n$ is vertex-transitive, so by Lemma~\ref{lem:transitive}(d) each of the $H_i$ graphs are vertex-transitive as well.  In particular, each $H_i$ is a graph on $|V(H)|$ vertices which is regular, and hence is regular of degree $\lam_{\max}(A_i)$.  This implies \[\sument(A_i)=\lam_{\max}(A_i) |V(H)|\] for $i\ge 1$, and this expression trivially holds for $i=0$ as well. Using this and \eqref{eq:rearrange} gives
			\begin{align*}\left(1+\frac{\lam_{\max}(A_n)}{|\lam_{\min}(A_n)|}\right)\cdot \sument((J-A_n)\circ Z)&=\left(1+\frac{\lam_{\max}(A_n)}{|\lam_{\min}(A_n)|}\right)\cdot \sum_{i=0}^{n-1} z_i  \sument(A_i)\\ 
				&= \sum_{i=0}^{n-1}z_i \lam_{\max}(A_i)|V(H)|+\frac{\lam_{\max}(A_n)}{|\lam_{\min}(A_n)|}\cdot \sum_{i=0}^{n-1}z_i \lam_{\max}(A_i)|V(H)|\\
				&\ge \sum_{i=0}^{n-1}z_i \lam_{\max}(A_i)|V(H)|+\frac{\lam_{\max}(A_n)}{|\lam_{\min}(A_n)|}\cdot z_n|\lam_{\min}(A_n)| |V(H)|\\ &=\sum_{i=0}^n z_i \lam_{\max}(A_n)|V(H)|=\sument(Z), \end{align*}
			proving the result.
		\end{proof}
		We can bootstrap Lemma~\ref{lem:ZJohnson} to prove an analogous result for matrices that are not necessarily in the span of the $A_i$ matrices.
		\begin{lemma}\label{lem:Z} Let $Z$ be a non-negative PSD matrix indexed by $V(H)$.  Then 
			\[\sument{(Z)}  \leq  \left(1+\frac{\lam_{\max}(A_n)}{|\lam_{\min}(A_n)|}\right)\sument{((J-A_n) \circ Z)}.\]
		\end{lemma}

		\proof For each permutation $\pi\in \Aut(H)$ of $V(H)$, let $P_\pi$ denote its corresponding permutation matrix indexed by $V(H)$.  Define
		\[
		\comp{Z} = \frac{1}{|\Aut(H)|} \sum_{\pi \in \Aut(H)} P_\pi^TZP_\pi.
		\]
		\begin{claim}
			We have $\sument(\comp{Z})=\sument(Z)$.
		\end{claim}
		\begin{proof}
			Observe that $\sument{(P^T Z  P)} =  \sument{(Z)}$ for any permutation matrix $P$ since permuting the terms in the sum does not change the result.  With this we see
			\begin{equation*}
				\sument(\comp{Z})=\frac{1}{|\Aut(H)|} \sum_{\pi \in \Aut(H)} \sument(P_\pi^TZP_\pi)=\sument(Z).
			\end{equation*}
		\end{proof}
		\begin{claim}
			We have $\sument((J-A_n)\circ \comp{Z})=\sument((J-A_n)\circ Z)$.
		\end{claim}
		\begin{proof}
			We first observe that for all $i\ge 0$ and $\pi\in \Aut(H)$, we have $P_\pi^T A_iP_\pi=A_i$.  This trivially holds for the identity matrix $A_0$, and for $i\ge 1$ this statement is equivalent to saying that each $\pi \in \Aut(H)$ is an automorphism of each $H_i$ which follows from Lemma~\ref{lem:transitive}(d).  In particular, since $\Aut(H)$ is closed under inverses and $J=\sum_{i=0}^n A_i$, this observation implies
			\[P_{\pi^{-1}}^T(J-A_n)P_{\pi^{-1}}=J-A_n\]
			for any $\pi\in \Aut(H)$.
			
			Now, for any $\pi\in \Aut(H)$, we have
			\begin{align}
				\sument{((J-A_n)\circ P_\pi^T Z P_\pi)} &= \tr((J-A_n)(P_\pi^T Z P_\pi)) \nonumber\\
				&= \tr(P_\pi (J-A_n)P_\pi^TZ) \nonumber\\
				&= \tr((J-A_n)Z) \nonumber\\
				&= \sument{((J-A_n)\circ Z)} \nonumber,
			\end{align}
			where the third equality used the observation made above and $P_\pi^T=P_{\pi^{-1}}$.   Summing this equality over all $\pi\in \Aut(H)$ and dividing by $|\Aut(H)|$ gives the desired result.
		\end{proof}
		\begin{claim}
			There exist real numbers $z_i$ such that $\comp{Z}=\sum_{i=0}^n z_i A_i$.
		\end{claim}
		\begin{proof}
			Let $u,v$ and $x,y$ be pairs of distinct vertices such that $uv,xy\in E(H_i)$ for some $i\ge 1$. By Lemma~\ref{lem:transitive}(c), there exists some $\sigma\in \Aut(H)$ with $\sigma(uv)=xy$, and by definition we have
			\[\comp{Z}_{x,y}=(P_\sigma^T\comp{Z} P_\sigma)_{u,v}=\left(\frac{1}{|\Aut(H)|}\sum_{\pi\in \Aut(H)} P_{\pi\circ \sigma}^T Z P_{\pi\circ \sigma}\right)_{u,v}= \left(\frac{1}{|\Aut(H)|}\sum_{\pi\in \Aut(H)} P_{\pi}^T Z P_{\pi}\right)_{u,v}=\comp{Z}_{u,v},\]
			where the third equality used that $\Aut(H)$ is a group so that summing over $\pi$ or $\pi\circ \sigma^{-1}$ gives the same quantity.
			
			The result above implies that for all $i\ge 1$ there exists some real number $z_i$ such that $Z_{u,v}=z_i$ whenever $(A_i)_{u,v}=1$. An analogous argument using that $H$ is vertex-transitive implies that there exists a $z_0$ with $Z_{u,u}=z_0$ for all $u$.  These two facts give the claim.
		\end{proof}
		Note that $\comp{Z}$, which is the sum of non-negative PSD matrices, is also non-negative and PSD, so the previous claim shows that Lemma~\ref{lem:ZJohnson} applies to $\comp{Z}$.  Using this and the other two claims gives
		\[\sument(Z)=\sument(\comp{Z})\le \left(1+\frac{\lam_{\max}(A_n)}{|\lam_{\min}(A_n)|}\right)\sument{((J-A_n) \circ \comp{Z})}= \left(1+\frac{\lam_{\max}(A_n)}{|\lam_{\min}(A_n)|}\right)\sument{((J-A_n) \circ Z)},\]
		proving the result.
		
		\qed

		We now prove Lemma~\ref{lem:main}, which we restate below.
		
		\begin{lemn}[\ref*{lem:main}]
			Let $G$ be a graph with an $H$-partition $\bigsqcup_{u\in V(H)} V_u$, and let $X$ be a PSD matrix with rows and columns indexed by $V(G)$.  If $H$ is vertex and edge-transitive, then
			\[\norm{X}^2\le \left(1+\frac{\lam_{\max}(H)}{|\lam_{\min}(H)|}\right)\sum_{(u,v):u,v\in V(H),\{u,v\}\notin E(H)} \norm{X_{[u,v]}}^2.\]
		\end{lemn}
		
		\begin{proof}
			Let $X$ be as in the statement of the lemma.  Let $Z$ be the matrix with rows and columns indexed by $V(H)$ such that
			\[Z_{u,v}=\norm{X_{[u,v]}}^2=\sument((X\circ X)_{[u,v]}).\]
			\begin{claim}
				The matrix $Z$ is non-negative and PSD.
			\end{claim}
			\begin{proof}
				Non-negativity is clear.  To show that it is PSD, consider any vector $\vec{z}$ indexed by $V(H)$.  Define a vector $\vec{x}$ indexed by $V(G)$ with $\vec{x}_w=\vec{z}_u$ whenever $w\in V_u$.  Because $X$ is PSD, $X\circ X$ is also PSD by the Schur product theorem.  Thus
				\[0\le \vec{x}^T(X\circ X) \vec{x}=\sum_{(u,v):u,v\in V(H)} \vec{z}_u \vec{z}_v \sument((X\circ X)_{[u,v]})=\vec{z}^TZ\vec{z}.\]
				As $\vec{z}$ was arbitrary, we conclude that $Z$ is PSD.
			\end{proof}
			
			From the definition of $Z$ we have
			
			\[\|X\|^2=\sum_{(u,v):u,v\in V(H)} \|X_{[u,v]}\|^2=\sum_{(u,v):u,v\in V(H)} Z_{u,v}=\sument(Z).\]
			With this it suffices to prove
			\begin{align}
				\sument{Z}=  &\leq \left(1+\frac{\lam_{\max}(H)}{|\lam_{\min}(H)|}\right)\sum_{(u,v):u,v\in V(H),\{u,v\}\notin E(H)}  Z_{u,v} \nonumber\\&=\left(1+\frac{\lam_{\max}(H)}{|\lam_{\min}(H)|}\right) \sument{((J-A_n) \circ Z)} \nonumber\end{align}
			The inequality above follows immediately from our claim and Lemma~\ref{lem:Z}, proving the result.
		\end{proof}

		\section{Further directions and open problems}\label{sec:conclusion}
		
		As noted in the introduction, it was asked by Anekstein, Elphick, and Wocjan \cite{wocjan2018more} whether the lower bound of Ando and Lin~\cite{AndLin2015} in Theorem~\ref{thm:AL} holds for the vector chromatic number $\chi_c(G)$.  In this paper we have taken the first step towards this question by showing that it holds for the fractional chromatic number $\chi_f(G)$.  While it is unlikely that our methods can bring us all the way down to proving bounds for $\chi_c(G)$, it is plausible that one can strengthen our Theorems~\ref{thm:main} and \ref{thm:mainGen} by considering other variants of the chromatic number defined in terms of homomorphisms.
		
		To formalize this idea, we let $\Phi(G)$ denote the set of graphs $H$ such that there exists a homomorphism from $G$ to $H$.  Given a set of graphs $\cH$, we define the \textit{$\cH$-chromatic number } of a graph $G$ by \[\chi_\cH(G):=1+\inf_{H\in \cH\cap \Phi(G)}\frac{\lam_{\max}(H)}{|\lam_{\min}(H)|}.\]
		For example, if $\cH$ consists of the set of cliques then $\chi_\cH(G)$ is exactly the chromatic number, and if $\cH$ is the set of Kneser graphs then $\chi_\cH(G)$ is the fractional chromatic number.  Similarly Theorem~\ref{thm:mainGen} can be viewed as a lower bound on $\chi_{\cH}(G)$ when $\cH$ is the set of edge-transitive graphs.  It would be of interest to study other bounds and properties of $\chi_{\cH}(G)$ for other families of graphs $\cH$.  In particular, we ask the following.
		\begin{question}\label{quest}
			What is the largest family of graphs $\cH$ such that for all graphs $G$,
			\[\chi_\cH(G)\ge 1+\max\left\{\frac{s^+(G)}{s^-(G)},\frac{s^-(G)}{s^+(G)}\right\}.\]
		\end{question}
		For example, Theorem~\ref{thm:mainGen} shows this holds when $\cH$ consists of all $H$ which are edge-transitive.  This result trivially extends to all $H$ for which there exists an edge-transitive graph $H'\in \Phi(H)$ such that $\lambda_{\max}(H)/|\lambda_{\min}(H)|\ge \lambda_{\max}(H')/|\lambda_{\min}(H')|$, so for example, this applies to all $H$ which are bipartite.  As discussed in Subsection~\ref{sec:comparison}, by considering $H=\overline{C_6}$, we know that Question~\ref{quest} can not hold for any $\cH$ containing all vertex-transitive graphs.
		
		There are other spectral problems where our approach may be fruitful.  In particular, Elphick and Wocjan \cite{ElpWoc2017} proved the following inertial bound for $\chi(G)$:
		\[\chi(G)\ge 1+\max\left\{\frac{n^+}{n^-},\frac{n^-}{n^+}\right\},\]
		where $n^+, n^-$ denote the numbers positive and negative eigenvalues of $G$, respectively. They conjectured that this same lower bound also holds for $\chi_f(G)$, and it is plausible that some of the ideas used here could be used to tackle this conjecture.  We note however that the stronger result \[\chi_{\cH}(G)\ge 1+\max\left\{\frac{n^+}{n^-},\frac{n^-}{n^+}\right\}\] does not hold when $\cH$ consists of all edge-transitive graphs (due to $G=C_5$ mapping to $C_5\in \cH$).  Thus to prove this conjectured lower bound for $\chi_f(G)$ using our approach, one would likely need to use stronger properties of the Johnson scheme than those guaranteed by Lemma~\ref{lem:transitive}.

		\section*{Acknowledgements}
		
		The authors gratefully acknowledge the American Institute of Mathematics; this collaboration started at \textit{Spectral Graph and Hypergraph Theory: Connections \& Applications}, December 6--10, 2021, a workshop at the American Institute of Mathematics. 
		The authors also are grateful for discussion on positive and negative squared energies with Aida Abiad, Leonardo de Lima, Dheer Noal Desai, Leslie Hogben and Jos\'{e} Madrid. 
		

	\end{document}